\documentclass[12pt,twoside]{amsart}
\usepackage{amsfonts}
\usepackage{amsmath}
\usepackage{amssymb}
\usepackage{mathrsfs}

\newtheorem{theo}{Theorem}
\newtheorem{cor}{Corollary}
\newtheorem{lem}{Lemma}

\theoremstyle{definition}

\theoremstyle{remark}

\numberwithin{equation}{section}

\def\1{{\mathchoice {\rm 1\mskip-4mu l} {\rm 1\mskip-4mu l}{\rm 1\mskip-4.5mu l} {\rm 1\mskip-5mu l}}}
\newcommand{\ds}{\displaystyle}

\usepackage[width=16.00cm, height=21.00cm, left=2.00cm, right=2.00cm, top=4cm, bottom=2cm]{geometry}

\title[Spectral properties of commutators  on harmonic Bergman spaces]{Spectral properties of commutators  on harmonic Bergman spaces of the unit disk}

\author[K. Chbichib and N. Ghiloufi]{Khaled Chbichib and Noureddine Ghiloufi}
\email{c.khaled@yahoo.fr, noureddine.ghiloufi@fsg.rnu.tn}
\address{University of Gabes\\ Faculty of Sciences of Gabes\\  Laboratory of Mathematics and Applications (LR17ES11)\\ 6072, Gabes, Tunisia.\vskip0.1cm
ORCID iD: https://orcid.org/0009-0003-7604-2143\vskip0.1cm
ORCID iD: https://orcid.org/0000-0001-8061-7702}

\begin{document}

\begin{abstract}
In this paper, we determine the asymptotic behavior of the singular values of the commutator $\mathscr{C}_u := [M_u, \mathbb{P}_\alpha]$ acting on $L^2(\mathbb D, dA_\alpha)$, where $M_u$ is the operator of multiplication by a subharmonic function $u$ on $\mathbb D(R)$ and harmonic outside the origin, with $R>1$, and $\mathbb{P}_\alpha$ is the orthogonal projection onto the space $\mathscr{H}_\alpha^2(\mathbb D)$ of harmonic functions on $\mathbb D$ that are square-integrable with respect to the weighted measure $dA_\alpha$. We prove that if $u(z) = U(z) + \overline{U(z)} + \nu_u \log|z|^2$, where $U$ is a holomorphic function on $\mathbb D(R)$ and $2\nu_u$ is the Lelong number of $u$ at $0$, then
$$s_n(\mathscr{C}_u) \underset{n\to+\infty}{\sim} \frac{\sqrt{\alpha+1}}{2\pi n} \int_{\partial \mathbb D} \sqrt{\nu_u^2 + |U'(z)|^2} \; |dz|.$$
In particular, the operator $\mathscr{C}_u$ belongs to the Von Newman-Schatten class $\mathcal C_p$ for any $p>1$.
\end{abstract}
\maketitle
\section{Introduction}
Let $\mathcal H$ be a Hilbert space and let $A$ and $B$ be two linear operators acting on $\mathcal H$. 
They are said to commute if $[A,B]:=AB-BA=0$. In this situation, the operators share important structural properties and can often be analyzed simultaneously. For instance, when $A$ and $B$ are diagonalizable operators, their commutativity ensures the existence of a common orthonormal basis of eigenvectors, with respect to which both operators are diagonal. \\
In contrast, when $[A,B]\neq 0$, such a simultaneous diagonalization property generally fails, and the commutator measures the extent to which the operators do not commute.\\

The study of commutators of operators plays a central role in several areas of mathematics and physics. 
A classical example arises in quantum mechanics, where the position and momentum operators satisfy the canonical commutation relation
$[A,B]=-\frac{i\hbar}{2\pi}I$, where $I$ denotes the identity operator and $\hbar$ is the reduced Planck constant. This fundamental identity leads to the well-known Heisenberg uncertainty principle.\\

Motivated by such relations, considerable attention has been devoted to the study of operator equations involving commutators. 
In particular, one is often interested in solving equations of the form $[A,X]=C,$ for given operators $A$ and $C$, or more generally the Sylvester-type equation $AX-XB=C,$ for prescribed operators $A$, $B$, and $C$. 
These equations arise naturally in many contexts of operator theory and potential theory and have been extensively investigated (see, for example, \cite{Me-Ma} and the references therein).\\

To introduce our statement, let $\alpha>-1$ and $dA_\alpha(z)=(\alpha+1)(1-|z|^2)^\alpha dA(z)$ be the probability measure on the unit disk $\mathbb D$ of $\mathbb C$ where $dA(z)=\frac{1}{\pi}dxdy=\frac{1}{\pi}rdrd\theta$ for $z=x+iy=re^{i\theta}\in\mathbb D$. This measure has many specific properties related to harmonic analysis on the unit disk (see \cite{HKZ} for more details). Let $\mathscr H_\alpha^2(\mathbb D)$ be the space of harmonic functions on $\mathbb D$ that are square-integrable with respect to the measure $dA_\alpha$. It is well known that $\mathscr H_\alpha^2(\mathbb D)$ is a sub-Hilbert space of $L^2(\mathbb D,dA_\alpha)$ and its reproducing kernel is given by $$K_\alpha(z,w)=\frac{1}{(1-z\overline{w})^{\alpha+2}}+\frac{1}{(1-w\overline{z})^{\alpha+2}}-1.$$
Hence the orthogonal projection $\mathbb P_\alpha$ from $L^2(\mathbb D,dA_\alpha)$ onto $\mathscr H_\alpha^2(\mathbb D)$ is given by
$$\mathbb P_\alpha f(z)=\langle f,K_\alpha(.,z)\rangle=\int_{\mathbb D}f(w)K_\alpha(z,w)dA_\alpha(w).$$

In \cite{Wu}, Wu considered the commutator $\mathscr C_u := M_u \circ \mathbb P_\alpha - \mathbb P_\alpha \circ M_u$ and the Hankel operator $H_u := (\mathbb I - \mathbb P_\alpha)\circ M_u$ associated with a harmonic function $u$ on $\mathbb D$ where $M_u$ denotes the multiplication operator by $u$. He proved that $\mathscr C_u$ (respectively, $H_u$) is compact on $L^2(\mathbb D, dA_\alpha)$ if and only if $|\nabla u(z)|(1-|z|^2)\longrightarrow 0$ as  $|z|\to 1^-.$
Moreover, he showed that the problem can be reduced to the case where $u$ is real-valued.\\
Later, in 2004, Dostani\'c focused on the analytic case and proved in \cite{Do} that the singular values of $\mathscr C_u$, acting on the Bergman space $\mathcal A^2(\mathbb D)$ of holomorphic functions that are square-integrable with respect to the Lebesgue measure on $\mathbb D$, satisfy the asymptotic estimate
$$s_n(\mathscr C_u)\underset{n\to+\infty}{\sim}\frac{1}{2\pi n}\int_{\partial \mathbb D}|u'(z)|\, |dz|,$$
whenever $u$ is holomorphic in a neighborhood of $\overline{\mathbb D}$.\\

In this paper, we consider a subharmonic function $u$ defined in a neighborhood $\mathbb D(R)$  of $\overline{\mathbb D}$ that is harmonic outside the origin. Since $u$ is bounded from above near $0$, it is well known that $u$ admits the representation $$u(z) = u_0(z) + \nu_u \log |z|^2,$$
where $u_0$ is harmonic in $\mathbb D(R)$ and $\nu_u$ denotes, up to a constant, the Lelong number of $u$ at $0$. Consequently, $u$ can be written as
$$
u(z) = U(z) + \overline{U(z)} + \nu_u \log |z|^2,
$$
where $U$ is a holomorphic function on $\mathbb D(R)$ (see \cite{Ra} for further background on harmonic and subharmonic functions).\\

Building on this representation, we extend Wu's analysis by establishing the asymptotic behavior of the singular values of the commutator $\mathscr{C}_u$. Indeed, Wu's results concern the case where $u$ is harmonic on $\mathbb D(R)$; that is, $\nu_u = 0$. This provides an answer to a question raised by Dostani\'c at the end of his paper~\cite{Do}. Our main result is stated below.
\begin{theo}
  Let $u$ be a subharmonic function on a neighborhood of $\overline{\mathbb D}$ with decomposition $u(z) = U(z) + \overline{U(z)} + \nu_u \log |z|^2$. Then the operator $\mathscr C_u$ is compact on $L^2(\mathbb D, dA_\alpha)$ and its singular values satisfy the following asymptotic estimate:
  $$s_n(\mathscr C_u)\underset{n\to+\infty}{\sim} \frac{\sqrt{\alpha+1}}{2\pi n}\int_{\partial \mathbb D}\sqrt{\nu_u^2+|U'(z)|^2}\;|dz|.$$
\end{theo}
The proof of the main result is carried out in two steps and relies on several auxiliary lemmas. Section~2 is devoted to recalling some well-known facts concerning the singular values of compact operators, along with preliminary results needed in the sequel. In Section~3, we then decompose the operator $\mathscr{C}_u$ as a sum of simpler operators and complete the proof.\\ 
As an immediate corollary, we deduce that $\mathscr{C}_u$ belongs to the Von Newman-Schatten class $\mathcal C_p$ for any $p>1$. That is, 
$$\sum_{n=0}^{+\infty} \left(s_n(\mathscr{C}_u)\right)^p<+\infty,\quad\forall\; p>1.$$
\section{Preliminary results}
We start by recalling several results related to compact operators, and then proceed to prove two preliminary lemmas that will be used throughout the paper.

\begin{itemize}
\item We first restrict ourselves to infinite-dimensional Hilbert spaces. Let $\mathcal H$ be an infinite-dimensional Hilbert space and let $T$ be a compact operator on $\mathcal H$. The sequence $(s_n(T))_{n}$ of eigenvalues of the operator
$$
|T|:=(T^*T)^{\frac12},
$$
arranged in non-increasing order, is called the sequence of singular values of $T$. Moreover, $T$ admits the following Schmidt decomposition:
$$
T=\sum_{n=1}^{+\infty}s_n(T)\langle \cdot, x_n\rangle y_n,
$$
where $(x_n)_n$ and $(y_n)_n$ are orthonormal sequences in $\mathcal H$. Furthermore, the adjoint operator $T^*$ is given by
$$
T^*=\sum_{n=1}^{+\infty}s_n(T)\langle \cdot, y_n\rangle x_n.
$$

\item Let $p\geq1$. We say that $T$ belongs to the von Neumann--Schatten class $\mathcal C_p$ if
$$
\sum_{n=1}^{+\infty}\bigl(s_n(T)\bigr)^p<+\infty.
$$
In this case, one defines the Schatten $p$-norm of $T$ by
$$
\|T\|_p:=\left(\sum_{n=1}^{+\infty}\bigl(s_n(T)\bigr)^p\right)^{\frac1p}.
$$
For further properties of the classes $\mathcal C_p$ and additional references, we refer, for instance, to the paper \cite{Me}.
\item (Birman--Solomyak, see \cite{BS}) Let $T$ be a compact operator on $L^2(\Omega,d\mu)$, where $\Omega$ is an open subset of $\mathbb R^d$, with integral kernel $\kappa(\cdot,\cdot)$:
$$Tf(t)=\int_\Omega \kappa(t,\theta)f(\theta)\,d\mu(\theta).$$
If the kernel $\kappa$ is smooth on $\overline{\Omega}\times\overline{\Omega}$, then there exists a constant $c>0$ such that
$$s_n(T)=O\big(e^{-c n^{1/d}}\big).$$
\item (Ky Fan-Weil inequalities, see \cite{GK}) Let $A$ and $B$ be compact operators on $\mathcal H$. Then
$$s_{n+m-1}(A+B)\leq s_n(A)+s_m(B), 
\qquad 
s_{n+m-1}(AB)\leq s_n(A)s_m(B).$$
In particular, if $B$ has finite rank $r$, then
$$s_{n+r}(A)\leq s_n(A+B)\leq s_{n-r}(A).$$
\item (Ky Fan asymptotic principle) If
$$\lim_{n\to+\infty} n^p s_n(A)=a
\quad \text{and} \quad
\lim_{n\to+\infty} n^p s_n(B)=0,$$
then
$$\lim_{n\to+\infty} n^p s_n(A+B)=a.$$
\item Let $A$ be a compact operator on $\mathcal H$, and let $(P_k)_k$ be a family of mutually orthogonal orthogonal-projections. Define
$$\widehat{A}=\sum_{k=1}^{+\infty} P_k A P_k.$$
Then $s_n(\widehat{A})=s_n(A)$ for all $n$ if and only if $A=\widehat{A}$.
\end{itemize}
Following this brief review of the singular values of compact operators, we proceed to study two operators that will be required in the proof of the main result. The first is the auxiliary operator $E$ defined as follows:
$$Ef(z)=\int_{\mathbb D}f(w)(z-w)^2K_\alpha(z,w)dA_\alpha(w),$$
while the second is $Q_0:L^2(\mathbb D,dA_\alpha)\longrightarrow L^2(\mathbb D,dA_\alpha)$ defined by
$$Q_0f(z)=\int_{\mathbb D}(z-w)f(w)K_\alpha(z,w)dA_\alpha(w).$$ 
\begin{lem}\label{l1}
    The singular values of $E$ satisfy
    $$s_n(E)\underset{n\to+\infty}{\sim} \frac{\sqrt{2(\alpha+1)(\alpha+4)}}{n^2}. $$
\end{lem}
\begin{proof}
    Using the formula
    $$K_\alpha(z,w)=1+\sum_{n=1}^{+\infty}\frac{(\alpha+2)_n}{n!}\left((z\overline{w})^n+(w\overline{z})^n\right)$$
    we obtain
    \begin{equation}\label{E}
    \begin{array}{ll}
    Ef(z)&=\ds\int_{\mathbb D}f(w)(z-w)^2dA_\alpha(w)+\sum_{n=1}^{+\infty}\frac{(\alpha+2)_n}{n!}\int_{\mathbb D}f(w)(z-w)^2\left((z\overline{w})^n+(w\overline{z})^n\right)dA_\alpha(w)\\
    &\ds=z^2\int_{\mathbb D}f(w)dA_\alpha(w)-2z\int_{\mathbb D}wf(w)dA_\alpha(w)+\int_{\mathbb D}f(w)dA_\alpha(w)+\sum_{j=1}^6I_jf(z)
    \end{array}
    \end{equation}
    where 
    \begin{equation}\label{E1}
    \begin{array}{ll}
    I_1f(z)&\ds=\sum_{n=1}^{+\infty}\frac{(\alpha+2)_n}{n!}z^{n+2}\int_{\mathbb D}f(w)(\overline{w})^ndA_\alpha(w)\\
    &\ds=\sum_{n=3}^{+\infty}\frac{(\alpha+2)_{n-2}}{(n-2)!}\langle f,w^{n-2}\rangle z^n,
    \end{array}
    \end{equation}
    \begin{equation}\label{E2}
    \begin{array}{ll}
    I_2f(z)&\ds=-2\sum_{n=1}^{+\infty}\frac{(\alpha+2)_n}{n!}z^{n+1}\int_{\mathbb D}wf(w)(\overline{w})^ndA_\alpha(w)\\
    &\ds=-2\sum_{n=2}^{+\infty}\frac{(\alpha+2)_{n-1}}{(n-1)!}\langle f,|w|^2w^{n-2}\rangle z^n,
    \end{array}
    \end{equation}
    \begin{equation}\label{E3}
    \begin{array}{ll}
    I_3f(z)&=\ds\sum_{n=1}^{+\infty}\frac{(\alpha+2)_n}{n!}z^{n}\int_{\mathbb D}w^2f(w)(\overline{w})^ndA_\alpha(w)\\
    &\ds=(\alpha+2)\langle f,|w|^2\overline{w}\rangle z+\sum_{n=2}^{+\infty}\frac{(\alpha+2)_n}{n!}\langle f,|w|^4w^{n-2}\rangle z^n,
    \end{array}
    \end{equation}
    \begin{equation}\label{E4}
    \begin{array}{ll}
    I_4f(z)&=\ds\sum_{n=1}^{+\infty}\frac{(\alpha+2)_n}{n!}z^2(\overline{z})^{n}\int_{\mathbb D}f(w)w^ndA_\alpha(w)\\
    &\ds=(\alpha+2)\langle f,\overline{w}\rangle z|z|^2+\sum_{n=0}^{+\infty}\frac{(\alpha+2)_{n+2}}{(n+2)!}\langle f,(\overline{w})^{n+2}\rangle |z|^4(\overline{z})^n,
    \end{array}
    \end{equation}
    \begin{equation}\label{E5}
    \begin{array}{ll}
    I_5f(z)&=-2\ds\sum_{n=1}^{+\infty}\frac{(\alpha+2)_n}{n!}z(\overline{z})^{n}\int_{\mathbb D}f(w)w^{n+1}dA_\alpha(w)\\
    &\ds=-2\sum_{n=0}^{+\infty}\frac{(\alpha+2)_{n+1}}{(n+1)!}\langle f,(\overline{w})^{n+2}\rangle |z|^2(\overline{z})^n
    \end{array}
    \end{equation}
    and 
    \begin{equation}\label{E6}
    \begin{array}{ll}
    I_6f(z)&=\ds\sum_{n=1}^{+\infty}\frac{(\alpha+2)_n}{n!}(\overline{z})^{n}\int_{\mathbb D}f(w)w^{n+2}dA_\alpha(w)\\
    &\ds=\sum_{n=1}^{+\infty}\frac{(\alpha+2)_n}{n!}\langle f,(\overline{w})^{n+2}\rangle (\overline{z})^n.
    \end{array}
    \end{equation}
    Combining Equations \eqref{E1}-\eqref{E6} with \eqref{E}, we obtain 
    \begin{equation}\label{E7}
    \begin{array}{lcl}
    Ef(z)&=&\ds \langle f,1-\overline{w}^2\rangle+\langle f,((\alpha+2)|w|^2-1)\overline{w}\rangle z+\langle f,\overline{w}\rangle ((\alpha+2)|z|^2-1)z\\
    &&\ds + \sum_{n=2}^{+\infty}\frac{(\alpha+2)_n}{n!}\left\langle f,\left(|w|^4-\frac{2n}{\alpha+n+1}|w|^2+\frac{n(n-1)}{(\alpha+n)(\alpha+n+1)}\right)w^{n-2}\right\rangle z^n\\
    &&\ds + \sum_{n=2}^{+\infty}\frac{(\alpha+2)_n}{(n+2)!}\langle f,(\overline{w})^n\rangle \left(|z|^4-\frac{2n}{\alpha+n+1}|z|^2+\frac{n(n-1)}{(\alpha+n)(\alpha+n+1)}\right)(\overline{z})^{n-2}.
    \end{array}
    \end{equation}
    If we set $$\varphi_n(z)=\sqrt{\frac{(\alpha+2)_n}{n!b_n}}\left(|z|^4-\frac{2n}{\alpha+n+1}|z|^2+\frac{n(n-1)}{(\alpha+n)(\alpha+n+1)}\right)z^{n-2},$$
    where 
    \begin{equation}\label{E8}
    \begin{array}{lcl}
      b_n&=&\ds\frac{(n+1)(n+2)}{(\alpha+n+2)(\alpha+n+3)}-4\frac{n(n+1)}{(\alpha+n+1)(\alpha+n+2)}+4\frac{n^2}{(\alpha+n+1)^2} -\frac{n(n-1)}{(\alpha+n)(\alpha+n+1)}\\
      &=&\ds \frac{2 (\alpha+1)((\alpha+4)n+\alpha^2+\alpha)}{(\alpha+n) (\alpha+n+1)^2 (\alpha+n+2) (\alpha+n+3)},
    \end{array}
    \end{equation}
    Then the sequences $(\varphi_n)_{n\geq2},\ (\overline{\varphi_n})_{n\geq2},\ (e_n)_{n\geq2}$ and $(\overline{e_n})_{n\geq2}$ are mutually orthonormal in $L^2(\mathbb D,dA_\alpha)$. Using these sequences, we can rewrite the formula \eqref{E7} as 
    $$\begin{array}{lcl}
        Ef(z)&=&\ds\langle f,1-\overline{w}^2\rangle+\langle f,((\alpha+2)|w|^2-1)\overline{w}\rangle z+\langle f,\overline{w}\rangle ((\alpha+2)|z|^2-1)z\\
         & &\ds+\sum_{n=2}^{+\infty}\sqrt{b_n}\langle f,\varphi_n\rangle e_n(z)+\sum_{n=2}^{+\infty}\sqrt{b_n}\langle f,\overline{e_n}\rangle \overline{\varphi_n(z)}.
      \end{array}$$
      It follows that the singular values of $E$ satisfy $s_n(E)\sim\sqrt{b_n}$ as $n\to+\infty$ and the lemma follows from the equivalence:
      $$b_n\underset{n\to+\infty}\sim \frac{2(\alpha+1)(\alpha+4)}{n^4}.$$
\end{proof}
    By duality, we deduce that the adjoint operator $E^*$ given by $$E^*f(z)=\int_{\mathbb D}f(w)(\overline{z}-\overline{w})^2K_\alpha(z,w)dA_\alpha(w)$$ satisfies the same estimate  $$s_n(E^*)\underset{n\to+\infty}{\sim} \frac{\sqrt{2(\alpha+1)(\alpha+4)}}{n^2}.$$
    For any $a\in\mathbb C$ and $\nu\geq0$, we consider the following operators: $\mathfrak{Q}_a=aQ_0+\overline{a}Q_0^*$, 
    $$\mathfrak{R}_\nu f(z)=\nu\int_{\mathbb D}(\log|z|^2-\log|w|^2)f(w)K_\alpha(z,w)dA_\alpha(w)\quad \text{and}\quad Y_{a,\nu}:=\mathfrak{Q}_a+\mathfrak{R}_\nu.$$
    Then, we have the following lemma:
\begin{lem}
    For any $a\in\mathbb C$ and $\nu\geq 0$, the singular values of the operator $Y_{a,\nu}$ satisfy the following asymptotic behavior: 
    $$s_n(Y_{a,\nu})\underset{n\to+\infty}{\sim}\frac{\sqrt{(\alpha+1)(\nu^2+|a|^2)}}{n}.$$
\end{lem}
\begin{proof}
    To prove the lemma, let us start with the study of $Q_0$. Using the same techniques as in the first step, it is straightforward to see that the operator $Q_0$ admits the following decomposition:
   $$\begin{array}{lcl}   
   Q_0f(z)&=&\ds -\sum_{n=0}^{+\infty}c_n\langle f,\phi_n\rangle e_{n+1}(z)+\sum_{n=0}^{+\infty}c_n\langle f,\overline{e_{n+1}}\rangle \overline{\phi_n(z)}\end{array}$$
   where $e_n$ is as above and 
   $$c_n=\sqrt{\frac{\alpha+1}{(\alpha+n+2)(\alpha+n+3)}},\quad \phi_n(w)=\sqrt{\frac{(\alpha+n+2)\Gamma(\alpha+n+4)}{(\alpha+1)\Gamma(\alpha+2)(n+1)!}}\left(|w|^2-\frac{n+1}{\alpha+n+2}\right)w^n$$
   for any $n\geq 0$.\\
   Let $a\in\mathbb C$ and $\nu\geq0$. Since the case $(a,\nu)=(0,0)$ is obvious, we may assume that $(a,\nu)\neq (0,0)$. Using the previous decomposition, we obtain 
   $$\begin{array}{lcl}   
   \mathfrak{Q}_af(z)&=&\ds -\sum_{n=0}^{+\infty}ac_n\langle f,\phi_n\rangle  e_{n+1}(z) +\sum_{n=0}^{+\infty}ac_n\langle f,\overline{e_{n+1}}\rangle \overline{\phi_n(z)}\\ 
   &&\ds -\sum_{n=0}^{+\infty}\overline{a}c_n\langle f, e_{n+1}\rangle \phi_n(z) +\sum_{n=0}^{+\infty}\overline{a}c_n\langle f,\overline{\phi_n}\rangle \overline{e_{n+1}(z)}.
   \end{array}$$
   It is easy to verify that the four sequences $(e_n)_{n\geq1},\ (\overline{e_n})_{n\geq1},\ (\phi_n)_{n\geq1}$ and $(\overline{\phi_n})_{n\geq1}$ are mutually orthonormal in $L^2(\mathbb D,dA_\alpha)$. \\
   
   For the second operator $\mathfrak{R}_\nu$, we have  
   $$\begin{array}{lcl}   
   \mathfrak{R}_\nu f(z)&=&\ds\nu\int_{\mathbb D}(\log|z|^2-\log|w|^2)f(w)K_\alpha(z,w)dA_\alpha(w)\\
   &=&\ds \sum_{n=0}^{+\infty}\langle f,e_n\rangle  \nu \log|z|^2e_n(z)+\sum_{n=0}^{+\infty}\langle f,\overline{e_n}\rangle \nu\log|z|^2\overline{e_n(z)}\\
   &&-\ds\sum_{n=0}^{+\infty}\langle f, \nu\log|w|^2e_n\rangle e_n(z) -\sum_{n=0}^{+\infty}\langle f,\nu\log|w|^2\overline{e_n}\rangle \overline{e_n(z)}.
   \end{array}$$
    Hence, by  combining the two previous equations, we find
    $$\begin{array}{l}   
   \ds Y_{a,\nu} f(z)=\langle f,1\rangle\nu\log|z|^2-\langle f,\nu\log|w|^2\rangle+\\
   \ds +\sum_{n=0}^{+\infty}\langle f, e_{n+1}\rangle\left(\nu\log|z|^2 e_{n+1}(z)-\overline{a}c_n \phi_n(z)\right) +\ds\sum_{n=0}^{+\infty}\langle f,\overline{e_{n+1}}\rangle\left(\nu\log|z|^2\overline{e_{n+1}(z)}-ac_n \overline{\phi_n(z)}\right)\\
   \ds-\sum_{n=0}^{+\infty}\langle f,\nu\log|w|^2e_{n+1}-\overline{a}c_n\phi_n\rangle  e_{n+1}(z)  -\sum_{n=0}^{+\infty}\langle f,\nu\log|w|^2\overline{e_{n+1}}-ac_n\overline{\phi_n}\rangle \overline{e_{n+1}(z)}.
   \end{array}$$
   It is easy to see that 
   $$x_n:=\langle \nu\log|z|^2 e_{n+1}-\overline{a}c_n \phi_n,e_{m+1}\rangle=\nu( \Psi'(n+1)-\Psi'(n+\alpha+2))\delta_{n,m},$$
where $\Psi$ denotes the digamma function,
$$\Psi(t)=\frac{\Gamma'(t)}{\Gamma(t)}
= -\gamma +\sum_{k=1}^{+\infty}\left(\frac{1}{k}-\frac{1}{k+t-1}\right),
\qquad t>0,$$
and $\gamma$ is the Euler--Mascheroni constant.\\
Hence, 
    $$\begin{array}{lcl}   
   \ds Y_{a,\nu} f(z)&=& \langle f,1\rangle\nu\log|z|^2-\langle f,\nu\log|w|^2\rangle\\
   &&\ds+\sum_{n=0}^{+\infty}\langle f, e_{n+1}\rangle\left((\nu\log|z|^2-x_n) e_{n+1}(z)-\overline{a}c_n \phi_n(z)\right) \\
   &&+\ds\sum_{n=0}^{+\infty}\langle f,\overline{e_{n+1}}\rangle\left((\nu\log|z|^2-x_n)\overline{e_{n+1}(z)}-ac_n \overline{\phi_n(z)}\right)\\
   &&\ds-\sum_{n=0}^{+\infty}\langle f,(\nu\log|w|^2-x_n)e_{n+1}-\overline{a}c_n\phi_n\rangle  e_{n+1}(z)\\
   &&\ds-\sum_{n=0}^{+\infty}\langle f,(\nu\log|w|^2-x_n)\overline{e_{n+1}}-ac_n\overline{\phi_n}\rangle \overline{e_{n+1}(z)}.
   \end{array}$$
    By putting $$t_n:=\sqrt{|a|^2c_n^2+\nu^2(\Psi'(n+2)-\Psi'(n+\alpha+3))^2}$$
    and 
    $$h_n(z):=\frac{1}{t_n}\left((\nu\log|z|^2-x_n)e_{n+1}(z)-\overline{a}c_n\phi_n(z)\right),$$
    we conclude that $(h_n)_{n\geq0},\ (e_{n+1})_{n\geq0},\ (\overline{h_n})_{n\geq0}$ and $(\overline{e_{n+1}})_{n\geq0}$ are  mutually orthonormal sequences in $L^2(\mathbb D,dA_\alpha)$ and 
    $$\begin{array}{lcl}   
   \ds Y_{a,\nu} f(z)&=&\ds \langle f,1\rangle\nu\log|z|^2-\langle f,\nu\log|w|^2\rangle+\sum_{n=0}^{+\infty}t_n\langle f, e_{n+1}\rangle h_n(z)\\
   &&\ds+\ds\sum_{n=0}^{+\infty}t_n\langle f,\overline{e_{n+1}}\rangle \overline{h_n(z)}-\sum_{n=0}^{+\infty}t_n\langle f,h_n\rangle  e_{n+1}(z)-\sum_{n=0}^{+\infty}t_n\langle f,\overline{h_n}\rangle \overline{e_{n+1}(z)}.
   \end{array}$$
   Hence, By Ky-Fan theorem, we deduce that $$s_n(Y_{a,\nu})\underset{n\to+\infty}{\sim}t_n\underset{n\to+\infty}{\sim}\frac{\sqrt{(\alpha+1)(\nu^2+|a|^2)}}{n}.$$   
\end{proof}   
 Now, for each $N\geq 1$, we split the unit disk into two parts: 
 $$D_0^N:=\{z\in\mathbb D;\ |z|<e^{-\frac{2\pi}{N}}\}, \qquad B_0^N:=\{z\in\mathbb D;\ e^{-\frac{2\pi}{N}}<|z|<1\}$$ and consider the corresponding multiplication operators $P_0:=M_{\1_{D_0^N}}$ and $P_1:=M_{\1_{B_0^N}}$ defined by multiplication by the characteristic functions of $D_0^N$ and $B_0^N$, respectively. Thus, the identity operator $I=P_0+P_1$ in $L^2(\mathbb D,dA_\alpha)$ and for any $a\in\mathbb C,\ \nu\geq0$, 
 $$Y_{a,\nu}=\left(P_0Y_{a,\nu}P_0+P_1Y_{a,\nu} P_0 +P_0Y_{a,\nu}P_1\right)+P_1Y_{a,\nu}P_1.$$ 
 Moreover, by Birman-Solomyak theorem, we obtain $$s_n(P_0Y_{a,\nu}P_0+P_1Y_{a,\nu}P_0+P_0Y_{a,\nu}P_1)=O(e^{-d_0\sqrt{n}})$$ as $n\to+\infty$ for some positive constant $d_0>0$. Hence, by Ky-Fan theorem and the previous lemma, we deduce that when $n\to+\infty$, $$s_n(P_1Y_{a,\nu}P_1)\sim s_n(Y_{a,\nu})\sim \frac{\sqrt{(\alpha+1)(\nu^2+|a|^2)}}{n}.$$
 To exploit this result in the proof of the main theorem, we decompose $P_1$ as a sum of $N$ operators $M_j$ defined by $M_j=M_{\1_{D_j^N}}$ where 
 $$D_j^N:=\left\{z\in\mathbb D;\ e^{-\frac{2\pi}{N}}<|z|<1,\ \frac{2\pi(j-1)}{N}<arg(z)<\frac{2\pi j}{N}\right\},\quad 1\leq j\leq N.$$
 We have the following lemma.
 \begin{lem}
     Let $a\in\mathbb C,\ \nu\geq0$ and $N\in\mathbb N$ be a positive integer. Then for every $1\leq j\leq N$, the singular values of the operator $M_jY_{a,\nu} M_j$ satisfy  the following asymptotic formula:
     $$s_n(M_jY_{a,\nu}M_j)\sim\frac{\sqrt{(\alpha+1)(\nu^2+|a|^2)}}{nN}.$$
 \end{lem}
 \begin{proof}
 Let $a\in\mathbb C$. Using the same notations as above, we obtain $P_1=\sum_{j=1}^N M_j$. Hence 
 $$P_1Y_{a,\nu}P_1=\sum_{1\leq j\neq k\leq N} M_jY_{a,\nu}M_k+ \sum_{j=1}^N M_jY_{a,\nu}M_j.$$
 Using the fact that $D_j^N\cap D_k^N=\emptyset$ for any $j\neq k$, it is not hard to see that 
 $$s_n(M_jY_{a,\nu}M_k)=o\left(\frac{1}{n}\right)$$ for $n\to+\infty$ and again by Ky-Fan inequality, we deduce that 
 $$s_n\left(\sum_{1\leq j\neq k\leq N} M_jY_{a,\nu}M_k\right)=o\left(\frac{1}{n}\right).$$
 That gives 
 $$s_n\left(\sum_{j=1}^N M_jY_{a,\nu}M_j\right)\sim \frac{\sqrt{(\alpha+1)(\nu^2+|a|^2)}}{n}.$$
Moreover, since $M_jY_{a,\nu}M_j$ is unitarily equivalent to $M_1Y_{a,\nu}M_1$, we conclude that $$s_n(M_jY_{a,\nu}M_j)\sim\frac{\sqrt{(\alpha+1)(\nu^2+|a|^2)}}{nN}$$ for every $1\leq j\leq N$.
 \end{proof} 
    
\section{Proof of the main result}

To prove the main result, we decompose the operator $\mathscr{C}_u$ into the sum of two operators, each admitting a decomposition in terms of its singular values. To this end, we claim that for every $f\in L^2(\mathbb D,dA_\alpha),$
$$\mathscr C_u f(z)=\int_{\mathbb D}f(w)(u(z)-u(w))K_\alpha(z,w)dA_\alpha(w).$$
Moreover, since $u$ is subharmonic on a disk $\mathbb D(R)$ for some $R>1$ and harmonic on $\mathbb D(R)\smallsetminus\{0\}$, there exists  a bounded holomorphic function $U$ on $\mathbb D(R)$ such that  $u=U+\overline{U}+\nu_u\log|z|^2$ where 
$$2\nu_u=\lim_{r\to0}\frac{\sup_{z\in\mathbb D(r)}u(z)}{\log(r)}$$ is the Lelong number of $u$ at 0.  Thus, one can find a holomorphic function $F_u$ on $\mathbb D(R)\times \mathbb D(R)$ such that for every $z,w \in\mathbb D(R)\smallsetminus\{0\},$
$$\begin{array}{lcl}
    \ds u(z)-u(w)&=&\nu_u\left(\log|z|^2-\log|w|^2\right)+U(z)-U(w)+\overline{U(z)-U(w)}\\
    &=&\ds\nu_u\left(\log|z|^2-\log|w|^2\right)+U'(w)(z-w)+(z-w)^2F_u(z,w)\\
    &&\ds+\overline{U'(w)(z-w)+(z-w)^2F_u(z,w)}.
\end{array}$$
Using this identity, the operator $\mathscr C_u$ can be decomposed as $\mathscr C_u=\mathfrak{R}_{\nu_u}+L_u+L_u^*+S_u+S_u^*$ where
$$L_uf(z)=\int_{\mathbb D}f(w)U'(w)(z-w)K_\alpha(z,w)dA_\alpha(w)$$
and
$$S_uf(z)=\int_{\mathbb D}f(w)F_u(z,w)(z-w)^2K_\alpha(z,w)dA_\alpha(w).$$

For the operator $S_u$, we have the following estimate:
    \begin{lem}
    The singular values of $S_u$ satisfy
    $$s_n(S_u)=o\left(\frac{1}{n}\right)\quad as\quad n\to+\infty.$$
    \end{lem}
    \begin{proof}
        Recall that $$S_uf(z)=\int_{\mathbb D}f(w)F_u(z,w)(z-w)^2K_\alpha(z,w)dA_\alpha(w).$$ 
        If we set $F_u^{(k)}(z,0)$ the $k^{th}-$derivative of the partial function $w\longmapsto F(z,w)$ at 0 for $z\in\mathbb D(R)$, then we obtain 
        $$F(z,w)=\sum_{k=0}^{+\infty}\frac{F^{(k)}(z,0)}{k!}w^k$$
        where the convergence is uniform on every compact subset of $\mathbb D(R)\times\mathbb D(R)$. The Cauchy integral formula gives that for $1<r<R$ and $z\in\mathbb D$ we have 
        $$\frac{F^{(k)}(z,0)}{k!}=\frac{1}{2i\pi}\int_{\partial\mathbb D(r)}\frac{F(z,w)}{w^{k+1}}dw.$$
        In particular, if we set $$N_r=\sup_{(\xi,w)\in\mathbb D\times \mathbb D(r)}|F(\xi,w)|,$$ 
        then we obtain $$\left|\frac{F^{(k)}(z,0)}{k!}\right|\leq \frac{N_r}{r^k}$$
        for every $z\in\mathbb D$. 
        Moreover, if $$Z_m(z,w)=(z-w)^2\sum_{k=m+1}^{+\infty}\frac{1}{k!}F^{(k)}(z,0)w^k,$$
        then it is easy to see that for every $z,w\in\overline{\mathbb D}$, we have $$|Z_m(z,w)|\leq \frac{N_r}{r^m(r-1)}.$$
        It follows that the operator $\mathcal Z_m $ defined by $$\mathcal Z_m f(z)=\int_{\mathbb D}f(w) Z_m(z,w)K_\alpha(z,w)dA_\alpha(w)$$
        is bounded on $L^2(\mathbb D,dA_\alpha)$ with norm $\|\mathcal Z_m\|\leq \frac{cN_r}{r^m(r-1)}$ for some positive constant $c>0$ independent of $m$. Again, since 
        $$\left\|M_{\frac{F^{(k)}(.,0)}{k!}}\right\|\leq \frac{N_r}{r^m}\quad\text{and}\quad\|M_{w^k}\|\leq 1,$$ then by setting $$\mathcal X_k=M_{\frac{F^{(k)}(.,0)}{k!}}\circ E\circ M_{w^k}$$
        we obtain $$S_u=\mathcal Z_m+\sum_{k=0}^m\mathcal X_k.$$
        Hence, by the previous inequalities and the properties of singular values of the sum of operators (Ky Fan - Weyl inequality), we find
        \begin{align*}
            s_{(m+2)n}(S_u)&\leq s_n(\mathcal Z_m)+\sum_{k=0}^m s_n(\mathcal X_k)\\
            &\leq \|\mathcal Z_m\|+\sum_{k=0}^m\frac{N_r}{r^k}s_n(E)\\
            &\leq \frac{cN_r}{r^m(r-1)}+\frac{rN_r}{r-1}s_n(E).
        \end{align*}
        Thanks to Lemma \ref{l1}, there exists a positive constant $\tau_0$ such that $s_n(E)\leq \frac{\tau_0}{n^2}$ for $m$ large enough. With the previous inequalities, we obtain 
        $$s_{(m+2)n}(S_u)\leq \frac{cN_r}{r^m(r-1)}+\frac{rN_r\tau_0}{(r-1)n^2}.$$
        By taking $$n=\left\lfloor \frac{2\log(m)}{\log(r)}\right\rfloor-1,$$
        we obtain $$s_{(m+2)n}(S_u)\leq \frac{rN_r}{(r-1)m^2}(rc+\tau_0).$$
        Thus, by putting $p=(m+2)n$ we find $p\sim \phi(n)$ where 
        $$\phi(x)=\frac{2}{\log(r)}x\log(x).$$
        We conclude that 
        $$s_p(S_u)\leq \frac{rN_r}{(r-1)}(rc+\tau_0)\frac{1}{(\phi^{-1}(p))^2}.$$
        In particular, for every $\varepsilon>0$, we have $$s_p(S_u)=O\left(p^{-\frac{2}{1+\varepsilon}}\right).$$
        We claim that if $W$ denotes the Lambert function, that is, the inverse of the map 
$t \mapsto t e^{t}$, then the above inverse function is given by
$$\phi^{-1}(p)=\exp\left(W\!\left(p\log(r)/2\right)\right).
$$
    \end{proof}
Now we can finish the proof of the main result. If $U$ is a holomorphic function in a neighborhood of $\overline{\mathbb D}$ as mentioned at the beginning for which $u=2\Re(U)$ then for any $\epsilon>0$, there exists an integer $N_0$ such that for any $N\geq N_0,\ 1\leq j\leq N$ and $w\in D_j^N$, we have $|U'(w)-U'(w_j)|<\epsilon$ where $w_j:=e^{i\eta_j}$ for some $\frac{2\pi (j-1)}{N}<\eta_j<\frac{2\pi j}{N}$. Using the same operators $M_j,\ 0\leq j\leq N$, we again obtain $$\mathfrak{R}_{\nu_u}+L_u+L_u^*=:T_u=\left(M_0T_uM_0+\sum_{0\leq j\neq k\leq N} M_jT_uM_k\right)+ \sum_{j=1}^N M_jT_uM_j$$ with 
 $$s_n\left(M_0T_uM_0+\sum_{0\leq j\neq k\leq N} M_jT_uM_k\right)=o\left(\frac{1}{n}\right).$$
Let us decompose $M_jT_uM_j=G_j^N+F_j^N$ with 
$$G_j^Nf(z)=\1_{D_j^N}(z)\int_{D_j^N}f(w)K_\alpha(z,w)\left((z-w)(U'(w)-U'(w_j))+ (\overline{z}-\overline{w})(\overline{U'}(w)-\overline{U'}(w_j))\right)dA_\alpha(w)$$
and 
$$F_j^Nf(z)=\1_{D_j^N}(z)\int_{D_j^N}f(w)K_\alpha(z,w)\left((z-w)U'(w_j)+ (\overline{z}-\overline{w})\overline{U'}(w_j)+\nu_u(\log|z|^2-\log|w|^2)\right)dA_\alpha(w).$$
Using the same technique used before, one can prove that 
$$s_n(G_j^N)\leq \frac{C\epsilon}{nN}$$ for some constant $C>0$ independent of $\epsilon,\ n$ and $N$. Since $G_j^N G_k^N=0$ for any $j\neq k$, we obtain 
$$\limsup_{n\to+\infty}ns_n\left(\sum_{j=1}^NG_j^N\right)\leq C\epsilon.$$
Thanks to the previous lemma, 
$$s_n(F_j^N)=s_n\left(M_jY_{U'(w_j),\nu_u}M_j\right)\underset{n\to+\infty}{\sim}\frac{\sqrt{(\alpha+1)(\nu_u^2+|U'(w_j)|^2)}}{nN}.$$
Again, using the fact that $F_j^N F_k^N=0$ for any $j\neq k$, we obtain 
$$s_n\left(\sum_{j=1}^NF_j^N\right)\underset{n\to+\infty}{\sim}\frac{\sqrt{\alpha+1}}{nN} \sum_{j=1}^N\sqrt{\nu_u^2+|U'(w_j)|^2}.$$
 Taking $N$ to infinity, we deduce that 
$$s_n(T_u)\underset{n\to+\infty}{\sim}\frac{\sqrt{\alpha+1}}{2\pi n}\int_0^{2\pi} \sqrt{\nu_u^2+|U'(e^{i\theta})|^2}d\theta.$$
That is,
$$s_n(\mathfrak{R}_{\nu_u}+L_u+L_u^*)\underset{n\to+\infty}{\sim}\frac{\sqrt{\alpha+1}}{2\pi n}\int_{\partial \mathbb D} \sqrt{\nu_u^2+|U'(z)|^2}\;|dz|.$$
 by the first step, $$s_n(S_u+S_u^*)=o\left(\frac{1}{n}\right).$$ Thus
$$s_n(\mathscr C_u)\underset{n\to+\infty}{\sim}\frac{\sqrt{\alpha+1}}{2\pi n}\int_{\partial \mathbb D} \sqrt{\nu_u^2+|U'(z)|^2}\;|dz|.$$ This completes the proof of the main result. 
\begin{cor}
    For any subharmonic function $u$ in a neighborhood $\mathbb D(R)$ of $\overline{\mathbb D}$ that is harmonic $\mathbb D(R)\smallsetminus\{0\}$, we have 
    \begin{enumerate}
        \item The operators $M_u$ and $\mathbb P_\alpha$ commute on $L^2(\mathbb D,dA_\alpha)$  if and only if $u$ is constant on $\mathbb D(R)$.
        \item The commutator $\mathscr C_u$ is in the von Newman-Schatten class $\mathcal C_p$ for any $p>1$.
    \end{enumerate}
\end{cor}
We conclude the paper by comparing the asymptotic behavior of the singular values of the commutator $\mathscr{C}_u$ with those of the Cauchy transform
$$
\mathcal C f(z)=-\int_{\mathbb D}\frac{f(w)}{w-z}\,dA_\alpha(w),
$$
which was investigated in a more general setting in \cite{CGS}. In contrast with the classical case $\alpha=0$, where the two operators have the same asymptotic behavior, we prove here, up to a multiplicative constant, that
$$
s_n(\mathscr{C}_u)\sim \frac{1}{n},
$$
whereas it was shown in \cite{CGS} that
$$
s_n(\mathcal C)\sim \frac{1}{n^{\alpha+1}}.
$$
\section*{Conflict of interest statement}
The authors declare that there is no conflict of interest.

 \end{document}